\documentclass{amsart}

\usepackage{amssymb}
\usepackage{delarray}
\usepackage[dvips]{graphics}
\usepackage{epsfig}
\usepackage{color}
\usepackage{dsfont}
\usepackage{comment}
\usepackage[super]{nth}

\newcommand{\C}{{\mathbb C}}
\newcommand{\Q}{{\mathbb Q}}
\newcommand{\Z}{{\mathbb Z}}
\newcommand{\F}{{\mathbb F}}
\newcommand{\N}{{\mathbb N}}
\newcommand{\R}{{\mathbb R}}

\newcommand{\Proj}{{\mathbb P}}

\newcommand{\res}{{\operatorname{res}}}
\newcommand{\nres}{{\operatorname{nres}}}
\newcommand{\comp}{{\operatorname{comp}}}
\newcommand{\sing}{{\operatorname{sing}}}

\newcommand{\wt}{{\operatorname{wt}}}
\newcommand{\SL}{{\operatorname{SL}}}

\newcommand{\QR}{{\operatorname{QR}}}
\newcommand{\CW}{{\operatorname{CW}}}

\newcommand{\PGL}{{\operatorname{PGL}}}
\newcommand{\Aut}{{\operatorname{Aut}}}

\newcommand{\tr}{{\operatorname{tr}}}
\newcommand{\legen}[2]{\left(\frac{#1}{#2}\right)}
\newcommand{\eps}{\varepsilon}
\newcommand{\es}[1]{\begin{equation}\begin{split}#1\end{split}\end{equation}}
\newcommand{\est}[1]{\begin{equation*}\begin{split}#1\end{split}\end{equation*}}
\newcommand{\ord}{\mathop{\rm ord}}
\newcommand{\eqdef}{\mathop{=}^{\rm def}}

\makeatletter
\let\@@pmod\pmod
\DeclareRobustCommand{\pmod}{\@ifstar\@pmods\@@pmod}
\def\@pmods#1{\mkern4mu({\operator@font mod}\mkern 6mu#1)}
\makeatother

\newtheorem{thm}{Theorem}

\newtheorem{lemma}{Lemma}
\newtheorem{prop}{Proposition}
\newtheorem{cor}{Corollary}
\newtheorem{example}{Example}

\title[Traces and refined weight enumerators]%
{Traces of Hecke operators and refined weight enumerators of Reed-Solomon codes}

\author{Nathan Kaplan}
\address{Department of Mathematics, University of California, Irvine, CA 92697-3875}
\email{nckaplan@math.uci.edu}

\author{Ian Petrow}
\address{\'Ecole Polytechnique F\'ed\'erale de Lausanne, Section des Math\'ematiques, 1015 Lausanne, Switzerland}
\email{ian.petrow@epfl.ch}
\thanks{The second author is partially supported by Swiss National Science Foundation grant 200021-137488, and an AMS-Simons travel grant.}


\begin{document}

\maketitle

\begin{abstract} 
We study the quadratic residue weight enumerators of the dual projective Reed-Solomon codes of dimensions $5$ and $q-4$ over the finite field $\F_q$.  Our main results are formulas for the coefficients of the the quadratic residue weight enumerators for such codes.  If $q=p^v$ and we fix $v$ and vary $p$ then our formulas for the coefficients of the dimension $q-4$ code involve only polynomials in $p$ and the trace of the $q$\textsuperscript{th} and $(q/p^2)$\textsuperscript{th} Hecke operators acting on spaces of cusp forms for the congruence groups $\SL_2 (\Z), \Gamma_0(2)$, and $\Gamma_0(4)$.  The main tool we use is the Eichler-Selberg trace formula, which gives along the way a variation of a theorem of Birch on the distribution of rational point counts for elliptic curves with prescribed $2$-torsion over a fixed finite field.
\end{abstract}

\section{Introduction}\label{intro}
The main goal of this paper is to show how traces of Hecke operators for the congruence subgroups $\SL_2(\Z), \Gamma_0(2)$ and $\Gamma_0(4)\cong \Gamma(2)$ enter into formulas for certain weight enumerators attached to classical and projective Reed-Solomon codes.  We study a refinement of the Hamming weight enumerator, which we call the \emph{quadratic residue weight enumerator}.   We prove a variation of the MacWilliams theorem for it, which follows in a straightforward way from the analogous MacWilliams theorem for the complete weight enumerator.

Projective and classical Reed-Solomon codes are maximum distance separable, which means that their Hamming weight enumerators are completely understood.  We compute the quadratic residue weight enumerator of the $5$-dimensional projective Reed-Solomon code, which leads directly to the corresponding weight enumerator of the $5$-dimensional classical Reed-Solomon code of length $q$ over the finite field $\F_q$.  By applying our version of the MacWilliams theorem we deduce formulas for individual coefficients of the quadratic residue weight enumerator of the projective Reed-Solomon code of dimension $q-4$ and for the corresponding classical Reed-Solomon code.  One of the main points of this paper is to demonstrate that there are interesting cases in which refined weight enumerators can be computed explicitly, giving additional information about rational point count distributions for varieties coming from well-studied codes. This addresses a particular case of Research Problem 11.2 of \cite{MacwilliamsSloane} about refined weight enumerators of Reed-Solomon codes.  

This paper fits into a literature about how weight enumerators of algebraically constructed codes can be expressed in terms of number-theoretic functions.  For a broad overview of these types of connections, focusing on codes and exponential sums over finite fields, see the survey of Hurt \cite{Hurt}.  Most directly related to our work is the detailed analysis of Zetterberg and Melas codes given in \cite{vanderGeerSchoofvanderVlugt, SchoofvanderVlugt}.  These codes are related to certain families of genus one curves over finite fields, and the Eichler-Selberg trace formula for $\Gamma_1(4)$ is the main tool in the proofs.  While these earlier families are considered only in characteristic $2$ and $3$, we consider codes in all characteristics not equal to $2$, and every isomorphism class of an elliptic curve over $\F_q$ contributes to the refined weight enumerators that we consider.  In \cite{SchoofSurvey}, Schoof notes that the appearance of Hecke operators acting on cusp forms for $\Gamma_1(4)$ in the formulas for these weight enumerators is ``probably related'' to the fact that the curves in the families considered have a rational point of order $4$, but does not give a ``direct connection".  In our analysis of the quadratic residue weight enumerator the connection to rational $2$-torsion points on elliptic curves is made much more explicit.  

We also note that the quadratic residue weight enumerator has appeared previously, for example in Section 5.8 of \cite{NebeRainsSloane} where it is used to study Hermitian self-dual codes over $\F_9$. However, we are unaware of any previous work connecting the coefficients of this weight enumerator to number theory.

\subsection{Reed-Solomon codes and weight enumerators}

Projective Reed-Solomon codes are constructed by evaluating each element of the $\F_q$-vector space of homogeneous polynomials of degree $h$ in two variables at an affine representative of each of the $q+1\ \F_q$-rational points of $\Proj^1$.  More precisely, take the standard choice of affine representatives $(1,a)$ where $a\in \F_q$ together with $(0,1)$ under some fixed ordering $p_1,\ldots, p_{q+1}$, and consider the evaluation map defined by
\[
f \mapsto \left(f(p_1),\ldots, f(p_{q+1})\right) \in \F_q^{q+1}.
\]
For $h \le q$ the image of this map is an $(h+1)$-dimensional linear subspace of $\F_q^{q+1}$ called the \emph{projective Reed-Solomon code}, or sometimes the doubly extended Reed-Solomon code, of order $h$.  A key observation is that a different choice of affine representatives gives an equivalent code, i.e., the same up to scaling and permuting coordinates.  We denote this code by $C_{1,h}$. Choosing a different ordering of the points also gives an equivalent code.  Puncturing such a code at one point, that is, deleting a fixed coordinate from each codeword, gives the classical, or affine, Reed-Solomon code of length $q$.

One reason why Reed-Solomon codes have received so much attention is that their minimal distance is as large as possible given their length and dimension. The \emph{Hamming distance} on $\F_q^n$ between $x = (x_1, \ldots, x_n)$ and $y = (y_1, \ldots, y_n)$ is defined
\est{d(x,y) \eqdef \#\left\{ i\in [1,n]\ \text{ \emph{such that} } x_i \neq y_i\right\},}
and the \emph{weight}, $\wt(x)$, of $x$ is the number of non-zero coordinates of $x$. That is, $\wt(\cdot)$ is a norm on $\F_q^n$ and the Hamming distance is the induced metric.  The \emph{minimal distance} of a code $C$ is $\min_{c_1 \neq c_2 \in C} d(c_1,c_2)$.  For a linear code, the minimal distance is equal to the minimal weight of a non-zero codeword.  The Singleton bound \cite[Ch.1, Thm 9]{MacwilliamsSloane}, states that the maximum size of a code $C\subseteq \F_q^n$ with minimum distance $d$ is $q^{n-(d-1)}$.  A code for which equality holds is called \emph{maximal distance separable} or \emph{MDS}.  A non-zero degree $h$ form on $\Proj^1$ vanishes at no more than $h$ distinct $\F_q$-rational points, implying that for $h \le q$ the code $C_{1,h}$ has minimum distance $q+1-h$ and is therefore MDS.

In order to analyze rational point count distributions for families of varieties, we would like to have a deeper understanding of the associated codes beyond their minimal distances.  The \emph{Hamming weight enumerator} of $C$ is a homogeneous polynomial in two variables that keeps track of the number of codewords of $C$ of each weight.  Given a code $C\subset \F_q^n$ we define
\[
W_C(X,Y)  \eqdef  \sum_{c \in C} X^{n-\wt(c)} Y^{\wt(c)} = \sum_{i=0}^n A_i X^{n-i} Y^i,
\]
where $A_i = \#\left\{ c\in C\ : \ \wt(c) = i\right\}$.  The discussion above shows that the weight enumerator of $C_{1,h}$ does not depend on the choice of affine representatives. The weight enumerator of an MDS code of length~$n$ over $\F_q$ is uniquely determined and easily computed. See Corollary 5 of Chapter 11 of \cite{MacwilliamsSloane}.

We study a refinement of the Hamming weight enumerator that carries additional information about the non-zero coordinates of codewords.  The \emph{quadratic residue weight enumerator} of  $C \subset\F_q$ is defined by
\begin{eqnarray*}
\QR_C(X,Y,Z)  \eqdef   \sum_{c\in C} X^{n-\wt(c)} Y^{\res(c)} Z^{\nres(c)}   =  \sum_{\substack{i,j,k \geq 0 \\ i+j+k=q+1}} A_{i,j,k} X^{i} Y^j Z^k,
\end{eqnarray*}
where $\res(c)$ denotes the number of coordinates of $c$ that are non-zero squares in $\F_q$, $\nres(c)$ denotes the number of coordinates that are not squares, and $A_{i,j,k}$ denotes the number of codewords $c \in C$ with $\res(c) = j$ and $\nres(c) = k$.  When $h$ is even $\QR_{C_{1,h}}(X,Y,Z)$ is well-defined since choosing a different affine representative for a projective point corresponds to multiplying the corresponding coordinate of each codeword by the same non-zero quadratic residue.

A main result (Theorem \ref{QRWt}) of this paper is the computation of $\QR_{C_{1,4}}(X,Y,Z)$ by a slight refinement of the methods of \cite{Deu} and \cite{Schoof}.  The coefficients of this weight enumerator solve an enumerative problem about elliptic curves.  The coefficient $A_{i,j,k}$ is equal to the number of homogeneous quartic polynomials $f_4(x,y)$ such that the variety defined by $w^2 = f_4(x,y)$ has exactly $i+2j\ \F_q$-rational points, $i$ of which come from roots of $f_4$.  This is related to counting elliptic curves over $\F_q$ with a specified number of rational points and a specified number of rational $2$-torsion points. See Section \ref{RSWEsec} for details.

In order to give the statement of our main result we introduce one additional important concept from coding theory, the dual code of a linear code.  Given $x = (x_1,\ldots, x_n)$ and $y = (y_1, \ldots, y_n)$ in $\F_q^n$ we define a non-degenerate symmetric bilinear pairing $\F_q^n \times \F_q^n \rightarrow \F_q$ by
\[
\langle x,y \rangle \eqdef \sum_{i=1}^n x_i y_i \in \F_q,
\]
and the \emph{dual code} of a linear code $C$ to be 
\[
C^\perp \eqdef \left\{y \in \F_q^n \ \mid \   \langle x,y \rangle = 0\ \forall\  x \in \F_q^n \right\}.
\]

The MacWilliams theorem \cite[Ch. 5, Thm 13]{MacwilliamsSloane}, says that the weight enumerator of $C$ determines the weight enumerator of $C^\perp$.
\begin{thm}[MacWilliams]\label{MacWthm}
Let $C \subseteq \F_q^n$ be a linear code.  Then
\[
W_{C^\perp}(X,Y) = \frac{1}{|C|} W_C(X+(q-1)Y,X-Y).
\]
\end{thm}
It is well-known that the dual of an MDS code is MDS, and moreover that the dual of a Reed-Solomon code is also a Reed-Solomon code.  More specifically, for $h\le q,\ C_{1,h}^\perp = C_{1,q-h-1}$.  See \cite[Ch. 11, Thm 2]{MacwilliamsSloane} for details. 

\subsection{Main results}

We use our computation of the quadratic residue weight enumerator of $C_{1,4}$ combined with a variation of the MacWilliams theorem to show that the coefficients of the quadratic residue weight enumerator of $C_{1,q-5}$ can be expressed in terms of traces of Hecke operators acting on spaces of cusp forms for the congruence subgroups $\SL_2(\Z),\Gamma_0(2),$ and $ \Gamma_0(4)$.  Let $\Q[x]$ denote the polynomial ring.  Let $v,R,N$ be integers, $v,N\geq 1$ and $R\geq 0$.  
Let $M_{v,R}(N)$ denote the set of functions on odd $v$\textsuperscript{th} prime powers $q = p^v$ of the form
\[
\sum_{k = 2 \atop k \equiv 0\pmod*{2}}^{2R+2} 
\left(g_{k,1}(p) \tr_{\Gamma_0(N),k} T_q +   g_{k,2}(p) \tr_{\Gamma_0(N),k} T_{q/p^2}\right),
\]
where each $g_{k,1}, g_{k,2} \in \Q[x]$ and we interpret $\tr_{\Gamma_0(N),k}T_{q/p^2}$ as $0$ unless $p^3 \mid q$.  We also let $M_{v}$ denote the set of functions on odd $v$\textsuperscript{th} prime powers $q=p^v$ of the form $g(p)$ for some $g\in \Q[x]$.  Note that the $M_{v,R}(N)$ and $M_v$ have the structure of $\Q[x]$-modules.  In what follows, all sums of modules take place within the $\Q[x]$-module of \emph{all} functions on odd $v$\textsuperscript{th} prime powers.  

\begin{thm}\label{MainTheorem} 
Let $v$ be a fixed positive integer. For fixed non-negative integers $j$ and $k$ consider the coefficient $A_{q+1-j-k,j,k}$ of $\QR_{C_{1,q-5}}(X,Y,Z)$ as a function on $v$\textsuperscript{th} powers of odd primes, $q = p^v$.   For each fixed residue class of $q \pmod{4},\ A_{q+1-j-k,j,k}$ is given by polynomials in $p$ and traces of Hecke operators.  More precisely, for each fixed $v,j,k$ and $\epsilon \in \{\pm 1\}$, there exists
\est{h \in M_v + M_{v,\lfloor \frac{j+k}{2}\rfloor}(1) + M_{v,\lfloor \frac{j+k-2}{2}\rfloor}(2) + M_{v,\lfloor \frac{j+k-3}{2}\rfloor}(4)} 
such that $A_{q+1-j-k,j,k} = h(q)$ for all $v$\textsuperscript{th} powers of odd primes $q = p^v$ satisfying $q\equiv \epsilon \pmod 4$.       
\end{thm}
\begin{example}
When $q=p\equiv 1 \pmod 4$ is a prime we find \begin{small}\est{  
\QR_{C_{1,q-5}}(X,Y,Z) = X^{q+1} + (q-1)^2q(q+1) \Bigg( \frac{1}{23040}(q^2 - 6q + 53)(q - 3) X^{q-5}\left(Y^6+Z^6\right) \\ 
+ \frac{1}{1536}(q - 1)(q - 3)(q - 5) X^{q-5}\left(Y^4Z^2+Y^2Z^4\right) \\ 
+ \frac{1}{645120}\left(q^5 - 20q^4 + 120q^3 - 860q^2 + 6154q - 13005- 35\tr_{\Gamma_0(4),6}T_q\right)X^{q-6}\left(Y^7+Z^7\right) \\
+ \frac{1}{92160}(q^5 - 20q^4 + 160q^3 - 660q^2  + 1274q - 765+ 5\tr_{\Gamma_0(4),6}T_q)X^{q-6}\left(Y^6Z+YZ^6\right) \\
+ \frac{1}{30720}(q^5 - 20q^4 + 160q^3 - 660q^2 + 1274q - 765+ 5\tr_{\Gamma_0(4),6}T_q )X^{q-6}\left(Y^5Z^2+Y^2Z^5\right) \\
+\frac{1}{18432}(q^5 - 20q^4 + 152q^3 - 508q^2+ 714q - 333 - 3\tr_{\Gamma_0(4),6}T_q ) X^{q-6}\left(Y^4Z^3+Y^3Z^4\right)\\ + O(X^{q-7})\Bigg).}\end{small}
\end{example}
\begin{example}
When $q=p\equiv 3 \pmod 4$ is a prime $\geq 7$ we find \begin{small}
\est{\QR_{C_{1,q-5}}(X,Y,Z) = X^{q+1} +(q-1)^2q(q+1)\Bigg( \frac{1}{3840}(q + 1)(q - 3)(q - 7)X^{q-5}\left(Y^5Z+YZ^5\right) \\ + \frac{1}{1152}(q^2 - 6q + 17)(q - 3)X^{q-5}Y^3Z^3 \\
 + \frac{1}{645120}(q^5 - 20q^4 + 120q^3 - 20q^2  - 566q - 405- 35\tr_{\Gamma_0(4),6}T_q)X^{q-6}\left(Y^7+Z^7\right) \\
+\frac{1}{92160}(q^5 - 20q^4 + 160q^3 - 540q^2 + 314q + 1035+ 5\tr_{\Gamma_0(4),6}T_q ) X^{q-6}\left(Y^6Z+YZ^6\right) \\
+ \frac{1}{30720}(q^5 - 20q^4 + 160q^3 - 540q^2  + 314q + 1035+ 5\tr_{\Gamma_0(4),6}T_q)X^{q-6}\left(Y^5Z^2+Y^2Z^5\right) \\
+ \frac{1}{18432}(q^5 - 20q^4 + 152q^3 - 628q^2  + 1674q - 2133- 3\tr_{\Gamma_0(4),6}T_q)X^{q-6}\left(Y^4Z^3+Y^3Z^4\right)\\ + O(X^{q-7})\Bigg).}\end{small}
\end{example}
Remarks: \begin{enumerate}
\item The above formulas match with an explicit brute-force computation of $\QR_{C_{1,q-5}}(X,Y,Z)$ for small values of $q$.
\item There are actions of $\Aut(\Proj^1(\F_q)) \cong \PGL_2(\F_q)$ and of $\F^*_q$ on non-zero codewords of $C_{1,q-5}$, so it is clear up to factors of 2 or 3 that $(q-1)^2q(q+1)$ divides all of the quadratic residue weight enumerator coefficients after the first.  
\item The denominators in the above example arise essentially only from the trinomial coefficients produced by the quadratic MacWilliams theorem.  If sufficiently motivated one could understand them completely. 
\item The proof of Theorem \ref{MainTheorem} (see the end of Section \ref{PointCountDistributions}) also gives other types of formulas, e.g. the case where one fixes $p$ and varies $v$; however the formulas involved are less aesthetically pleasing.
\end{enumerate}

Similar ideas can be used to show analogous results for the weight enumerators of classical Reed-Solomon codes of dimension $q-5$.  The dual of the Reed-Solomon code of dimension $5$ and length $q$ over $\F_q$ is the Reed-Solomon code of dimension $q-5$ and length $q$, which we denote $C'_{1,q-5}$.  
\begin{example}
Let $a(q)$ be the $q$\textsuperscript{th} Fourier coefficient of $\eta^{12}(2z)$, the unique normalized Hecke eigenform of weight $6$ for $\Gamma_0(4)$.  

When $q \ge 11$ is a prime congruent to $1$ modulo $4$, the $X^{q-7} Y^7$ coefficient of $\QR_{C'_{1,q-5}}(X,Y,Z)$ is 
\est{ \frac{1}{645120} {(q-6)q(q-1)^2 (q^5-20q^4 +120 q^3 -860 q^2 +6154 q - 13005)} 
\\ -  \frac{1}{18432} (q-6)q(q-1)^2 a(q). }
When $q \ge 7$ is a prime congruent to $3$ modulo $4$, the $X^{q-7} Y^7$ coefficient of $\QR_{C'_{1,q-5}}(X,Y,Z)$ is 
\est{ \frac{1}{645120} (q-6)q (q+1) (q-1)^2 (q^4-21q^3 +141 q^2 - 161 q -405) \\
- \frac{1}{18432} (q-6)q(q-1)^2 a(q).}
\end{example}
Note that up to factors of 2 or 3 that $q(q-1)^2$ divides all of the coefficients of $\QR_{C'_{1,q-5}}(X,Y,Z)$ after the first since there is an action on non-zero codewords by the subgroup of $\Aut(\Proj^1(\F_q))$ fixing a point of $\Proj^1(\F_q)$, as well as an action of $\F^*_q$ by scaling.  

The results of Section \ref{PointCountDistributions} can also be used to give exact formulas for certain sums involving rational point counts for families of elliptic curves over a fixed finite field. 
\begin{example}
Let $p$ be an odd prime.  For $a,b \in \F_p$, let $E_{a,b}$ denote the projective curve $y^2 z = x(x^2+axz+bz^2)$. We define
\[
S'_3(p) = \sideset{}{'}\sum_{a,b \in \F_p} \left(\#E_{a,b}(\F_p) - (p+1)\right)^6
\]
where the symbol $\sum'$ indicates that we only sum over pairs $(a,b)$ such that $E_{a,b}$ defines an elliptic curve.  Then
\[
S'_3(p) = (p-1)(p+1)\left(5p^3 -10 p^2 - 8p - 2\right) - \frac{1}{2} (p-1) \tr_{\Gamma_0(4),8}T_p.
\]
\end{example}
\noindent Birch gives similar formulas for sums taken over all elliptic curves over $\F_p$ in \cite{Birch}.

The projective Reed-Solomon codes $C_{1,h}$ fit into a broader class of projective Reed-Muller codes. One analogously defines $C_{n,h}$ by evaluating each homogeneous degree $h$ form on $\Proj^n$ at each of the $(q^{n+1}-1)/(q-1)\ \F_q$-rational points of $\Proj^n$.  For $n>1$ these codes are not MDS and their Hamming weight enumerators are generally quite hard to compute.  For example, the weight enumerators from codes from quadrics in $\Proj^n,\ C_{n,2}$, are computed in \cite{Elk}, along with the weight enumerator from codes coming from plane cubics, $C_{2,3}$, and from cubic surfaces, $C_{3,3}$.  The $C_{2,3}$ case is most similar to the results of this paper.  The weight enumerator is given in terms of the sizes of isogeny classes of elliptic curves over $\F_q$ and the results of Birch described in Section \ref{PointCountDistributions} show that the coefficients of the weight enumerator of $C_{2,3}^\perp$ can be expressed in terms of traces of Hecke operators acting on cusp forms for $\SL_2(\Z)$.  One of the interesting aspects of our main result is that by treating rational $2$-torsion points differently we also get contributions from the congruence subgroups $\Gamma_0(2)$ and $\Gamma_0(4)$.

Projective Reed-Solomon and Reed-Muller codes give examples of a more general construction of codes from evaluating polynomials at the $\F_q$-rational points of projective varieties.  This evaluation construction has been extensively studied by Tsfasman and Vl{\u{a}}du{\c{t}}, Lachaud, S{\o}rensen, and others \cite{Lachaud,Sor,TV}.  For much more information, particularly focusing on codes from higher-dimensional varieties, see the survey of Little \cite{Little}.

\section{Weight enumerators of codes from genus one curves}\label{genus1curves}

This section has two parts.  In the first we compute the quadratic residue weight enumerator $\QR_{C_{1,4}}(X,Y,Z)$ of the $5$-dimensional projective Reed-Solomon code.  In the second we give a variation of the classical MacWilliams theorem for this quadratic residue weight enumerator.  

\subsection{The quadratic residue weight enumerator of $C_{1,4}$}\label{RSWEsec}

For non-negative integers $i,j,k$ with $i+j+k=q+1$ let $A_{i,j,k}$ be the number of homogeneous quartics $f_4(x,y)$ that have $i\ \F_q$-rational roots and such that the variety defined by $w^2 = f_4(x,y)$ has exactly $i+2j\ \F_q$-points.  We define the quadratic residue weight enumerator to be
\[
\QR_{C_{1,4}}(X,Y,Z) \eqdef \sum_{\substack{i,j,k\geq 0 \\ i+j+k=q+1}} A_{i,j,k} X^{i} Y^j Z^k.
\]
We compute these coefficients by building slightly on work of Deuring on elliptic curves over a fixed finite field \cite{Deu}.

We first consider quartics $f_4(x,y)$ that have a double root.  In this case, $w^2 = f_4(x,y)$ is singular and counting points on this variety is elementary.  We then compute $\QR_{C_{1,4}}(X,X^2,1)$ which gives the rational point count distribution for the family of varieties being considered, but does not distinguish between points that come from $\F_q$-rational roots of $f_4(x,y)$ and points of $\Proj^1(\F_q)$ at which $f_4(x,y)$ takes a non-zero quadratic residue value.  Finally, we consider elliptic curves with a given number of points and prescribed $2$-torsion structure to compute $\QR_{C_{1,4}}(X,Y,Z)$.

\begin{prop}\label{QRsing}
Let $\QR^{\sing}_{C_{1,4}}(X,Y,Z)$ denote the contribution to $\QR_{C_{1,4}}(X,Y,Z)$ from quartics $f_4(x,y)$ that do not have distinct roots over $\overline{\F}_q$.  Then $\QR^{\sing}_{C_{1,4}}(X,Y,Z)$ is given by 
\begin{small}\begin{eqnarray*}
& & X^{q+1} + \frac{(q-1)(q+1)}{2} X (Y^q+Z^q) + (q-1) q (q+1) X^2 Y^{\frac{q-1}{2}} Z^{\frac{q-1}{2}} \\
&  + & \frac{(q-1)q(q+1)}{4} X^2 (Y^{q-1} + Z^{q-1})  +  \frac{(q-1)^2 q (q+1)}{4} X^3 (Y^{\frac{q-1}{2}} Z^{\frac{q-3}{2}} + Y^{\frac{q-3}{2}} Z^{\frac{q-1}{2}}) \\
& + & \frac{(q-1)^2 q(q+1)}{4} X (Y^{\frac{q+1}{2}} Z^{\frac{q-1}{2}} + Y^{\frac{q-1}{2}} Z^{\frac{q+1}{2}} ) +  \frac{(q-1)^2 q}{4} (Y^{q+1} + Z^{q+1}).
\end{eqnarray*}\end{small}
\end{prop}

\begin{proof}
There is a small list of factorization types of quartics with a double root.  Such a quartic could have a quadruple root, a root of multiplicity three and another rational root, two distinct double roots, or a double root and two other roots.  We work out the details of this last case and leave the others as an exercise.

A quartic with a single double root must have its double root at an \hbox{$\F_q$-rational} point.  The other two roots can then either be at distinct rational points or be a Galois-conjugate pair of points defined over $\F_{q^2}$.  Scaling a quartic by a quadratic residue does not change its contribution to $\QR_{C_{1,4}}(X,Y,Z)$, while scaling by a quadratic non-residue interchanges the number of residue versus non-residue values taken.  We write such a quartic as $f(x,y)^2 g(x,y)$, where $g(x,y)$ is a quadratic polynomial with distinct roots and $f(x,y)$ is a linear form with an $\F_q$-rational root.  The curve $w^2 = g(x,y)$ is a smooth conic, so has $q+1$ rational points.  Therefore, $w^2 = f(x,y) g(x,y)$ has either $q$ or $q+2\ \F_q$-points depending on the value taken by $g(x,y)$ at the rational root of $f(x,y)$.  Combining these observations shows that the contribution to $\QR_{C_{1,4}}(X,Y,Z)$ from quartics with a double root and two other distinct roots is
\[
\frac{(q-1)^2q(q+1)}{4}  \left( X^3 (Y^{\frac{q-1}{2}} Z^{\frac{q-3}{2}} + Y^{\frac{q-3}{2}} Z^{\frac{q-1}{2}}) + X (Y^{\frac{q+1}{2}} Z^{\frac{q-1}{2}} + Y^{\frac{q-1}{2}} Z^{\frac{q+1}{2}} )\right).
\] 
\end{proof}

We now turn to $\QR_{C_{1,4}}(X,X^2,1)$.  We first compute the number of times that a particular isomorphism class of an elliptic curve arises as an equation of the form $w^2 = f_4(x,y)$.

\begin{prop}\label{QuartDC}
Let $E$ be an elliptic curve defined over $\F_q$.  The number of homogeneous quartic polynomials $f_4(x,y)$ such that $w^2 = f_4(x,y)$ gives a curve isomorphic to $E$ is 
\[(q-1) \frac{|\PGL_2(\F_q)|}{|\Aut_{\F_q}(E)|} = \frac{(q-1)^2 q(q+1)}{|\Aut_{\F_q}(E)|}.\]
\end{prop}

\begin{proof}
We will phrase this as a double counting argument.  Suppose we begin with an elliptic curve $E$ with $\#E(\F_q) = q+1-t$.  There are exactly $q+1-t$ choices of a degree two divisor class on $E$.  Riemann-Roch implies that such a divisor has a $2$-dimensional space of sections.  Choosing a basis for this space of sections gives a degree $2$ map to $\Proj^1$.  Taking the inverse image of a point in $\Proj^1(\F_q)$ recovers the divisor class.  The branch points of this map are the roots of this quartic.

Now we consider how many maps take a particular equation of the form $w^2 = f_4(x,y)$ to the underlying elliptic curve $E$.  We can recover $E$ with a distinguished identity element and a degree $2$ divisor class $D$ directly from this equation.  Now we take a map that forgets $D$, taking $(E,D)$ to $E$, and note that it is defined only up to an automorphism of $E$ defined over $\F_q$.  Since an automorphism must fix the identity element of $E$, we multiply $|\Aut_{\F_q}(E)|$ by the number of possible choices of identity element, $q+1-t$.  Therefore, given $E$ there are 
\[
\frac{(q+1-t) (q-1) |\PGL_2(\F_q)|}{|\Aut_{\F_q}(E)| (q+1-t)} = \frac{(q-1)|\PGL_2(\F_q)|}{|\Aut_{\F_q}(E)|}
\]
quartics $f_4(x,y)$ with $w^2 = f_4(x,y)$ isomorphic to $E$.
\end{proof}

Let $N(t)$ be the number of $\F_q$-isomorphism classes of elliptic curves within the isogeny class $I(t)$ of curves having $\#E(\F_q) = q+1-t$, and let $N_A(t)$ be the number of $\F_q$-isomorphism classes of elliptic curves in $I(t)$ where each isomorphism class is weighted by $1/|\Aut_{\F_q}(E)|$.  Let $\QR^{S}_{C_{1,4}}(X,Y,Z)$ be the contribution to $\QR_{C_{1,4}}(X,Y,Z)$ from quartics with distinct roots over $\overline{\F}_q$ so that 
\est{ 
\QR_{C_{1,4}}(X,Y,Z) = \QR^{S}_{C_{1,4}}(X,Y,Z) + \QR^{\text{sing}}_{C_{1,4}}(X,Y,Z).
}
\begin{cor}\label{QRNt}
Suppose $q$ is odd.  
We have
\[
\QR^{S}_{C_{1,4}}(X,X^2,1) = \sum_{t^2\leq 4q} N_A(t) (q-1)^2 q (q+1) X^{q+1-t}.
\] 
\end{cor}

We now give expressions for $N_A(t)$ in terms of class numbers of orders in quadratic imaginary fields.  For $d<0$ with $d \equiv 0,1\pmod{4}$, let $h(d)$ denote the class number of the unique quadratic order of discriminant $d$.  Let
\est{h_w(d)\eqdef \begin{cases} h(d)/3, &\text{ if } d = -3; \\ h(d)/2, & \text{ if } d=-4; \\ h(d) & \text{else}, \end{cases}} 
and 
\est{H_w(\Delta) \eqdef \sum_{\substack{d^2 \mid \Delta \\ \Delta/d^2 \equiv 0,1 \pmod* 4}} h_w(\Delta/d^2)} 
be the Hurwitz-Kronecker class number.  

The following result is stated for prime fields $\F_p$ where $p \ge 5$, with slightly different notation, on page 654 of Lenstra's paper \cite{Lenstra}, in which he says that it is basically due to Deuring \cite{Deu}.  The details for the extension to all finite fields are contained in Chapter 4, particularly Theorem 4.5, of the paper of Waterhouse \cite{Waterhouse}.  Schoof gives an unweighted version as Theorem 4.6 in \cite{Schoof}.
\begin{thm}[Sizes of Isogeny Classes with Weights]\label{S46withweights}
Let $t\in \Z$.  Suppose $q = p^v$ where $p \neq 2$ is prime.  Then
\begin{eqnarray*}
2 N_A(t)  = &  H_w(t^2-4q)\ \ \ \ \ \ \ \ \ \ \ \ \ \ \ & \text{if } t^2 < 4q \text{ and } p\nmid t;\\
 = & H_w(-4p) \ \ \ \ \ \ \ \ \ \ \ \ \ \ \  & \text{if } t=0 \\
 = & 1/3 \ \ \ \ \ \ \ \ \ \ \ \ \ \ \  & \text{if } t^2=3q \text{ and } p =3 
\end{eqnarray*}
if $q$ is not a square, and 
\begin{eqnarray*}
2 N_A(t)  = &  H_w(t^2-4q)\ \ \ \ \ \ \ \ \ \ \ \ \ \ \ & \text{if } t^2 < 4q \text{ and } p\nmid t;\\
 = & \left(1 - \legen{-4}{p}\right)/2 \ \ \ \ \ \ \ \ \ \ \ \ \ \ \  & \text{if } t=0 \\
 = & \left(1 - \legen{-3}{p}\right)/3 \ \ \ \ \ \ \ \ \ \ \ \ \ \ \  & \text{if } t^2 = q \\
  = & (p-1)/12\ \ \ \ \ \ \ \ \ \ \ \ \ \ \  & \text{if } t^2 = 4q
\end{eqnarray*}
if $q$ is a square, and $N_A(t) = 0$ in all other cases.
\end{thm}

As above, let $N_{A, 2\times2}(t)$ denote the number of $\F_q$-isomorphism classes of elliptic curves over $\F_q$ with $E(\F_q)[2] \cong \Z/2\Z \times \Z/2\Z$ and where each class is weighted by the size of the automorphism group of a curve in that class.  The count of unweighted isomorphism classes of elliptic curves with full \hbox{$2$-torsion} is given in Lemma 4.8 of Schoof \cite{Schoof}.  Similar arguments give the weighted version, or paying particular attention to curves of $j$-invariant $0$ and $1728$ with $|\Aut_{\F_q} E| > 2$ we can easily deduce the weighted statement from the unweighted one.

\begin{lemma}\label{Schoof1}
Let $q = p^v$ where $p \neq 2$ is prime. Suppose that $t\in \Z$ satisfies $t^2 \le 4q$.
\begin{enumerate}
\item If $p\nmid t$ and $t\equiv q+1 \pmod{4}$, then 
\[
2 N_{A,2\times 2}(t) = H_w\left(\frac{t^2 - 4q}{4}\right).
\]
\item If $t^2 = q,2q$, or $3q$, then $N_{A,2 \times 2}(t) = 0$.
\item If $t^2 = 4q$ then $N_{A,2 \times 2}(t) = N_A(t)$.
\item Let $t=0$. If $q\equiv 1 \pmod{4}$ then $N_{A,2 \times 2}(t) = 0$.  If $q \equiv 3\pmod{4}$ then $2N_{A,2\times 2}(t) = h_w(-p)$. 
\end{enumerate}
Otherwise we have $N_{A,2 \times 2}(t)=0$.
\end{lemma}

We now turn to the full computation of $\QR_{C_{1,4}}(X,Y,Z)$. 
The main problem we need to solve is the following.  Suppose that there are $M$ smooth quartics $f_4(x,y)$ such that $w^2 = f_4(x,y)$ has exactly $q+1-t\ \F_q$-points.  Let $M_k$ be the number of these quartics with $k\ \F_q$-rational roots, so $M_0+M_1+M_2 + M_3 + M_4 = M$.  We need the individual values of the $M_i$.  If a quartic $f_4(x,y)$ defined over $\Proj^1(\F_q)$ has $4$ distinct roots and $3$ of them are $\F_q$-rational then the fourth root is also $\F_q$-rational.  Therefore, $M_3 = 0$ and we can determine $M_1$ using a very elementary observation.

\begin{lemma}
Suppose that $q+1-t$ is odd and that there are $M$ smooth quartics $f_4(x,y)$ such that $w^2 = f_4(x,y)$ has exactly $q+1-t\ \F_q$-points.  Then $M_1 = M$ and $M_0 = M_2 = M_4 = 0$.
Suppose that $q+1-t$ is even and that there are $M$ smooth quartics $f_4(x,y)$ such that $w^2 = f_4(x,y)$ has exactly $q+1-t\ \F_q$-points.  Then $M_1 = 0$.
\end{lemma}
\begin{proof}
The number of $\F_q$-rational points of $w^2 = f_4(x,y)$ is the number of $\F_q$-rational roots of $f_4(x,y)$ plus twice the number of points of $\Proj^1(\F_q)$ for which the quartic takes a non-zero square value.  Therefore, if $q+1-t$ is odd, then the number of roots of $f_4(x,y)$ is odd.  If $q+1-t$ is even, then the number of roots of $f_4(x,y)$ is even.
\end{proof}

We suppose that $q+1-t$ is even and determine how these $M$ quartics break up into those that have $0,2$, and $4\ \F_q$-rational roots.  We first note that for an elliptic curve in affine Weierstrass form $y^2 = f(x) =  x^3 + ax + b$, the roots of the homogeneous quartic $y(x^3+ax y^2 + by^3)$ give the $2$-torsion points of $E$.  When we consider curves given by $w^2 = f_4(x,y)$, a homogeneous quartic on $\Proj^1(\F_q)$ there is a similar correspondence between roots of $f_4(x,y)$ and $2$-torsion points of $E$.

\begin{lemma}\label{2tors}
Let $E$ be an elliptic curve defined over $\F_q$ and suppose that there are $M$ quartics $f_4(x,y)$ with $w^2 = f_4(x,y)$ isomorphic to $E$.  Let $M = M_0 + M_2 + M_4$, where $M_k$ is the number of quartics with $k\ \F_q$-rational roots.
\begin{enumerate}
\item If $E(\F_q)[2] \cong \Z/2\Z$ then $M_0 = M_2 = \frac{M}{2}$ and $M_4 = 0$.
\item If $E(\F_q)[2] \cong \Z/2\Z \times \Z/2\Z$ then $M_0 = \frac{3M}{4}, M_2 = 0$, and $M_4 = \frac{M}{4}$.
\end{enumerate}
\end{lemma}
\begin{proof}
We describe how to find all quartics $f_4(x,y)$ with $w^2 = f_4(x,y)$ isomorphic to $E$. Riemann-Roch implies that a degree $2$ divisor on $E$ has a $2$-dimensional space of sections.  Given such a divisor, choosing a basis for this space of sections gives a degree $2$ map to $\Proj^1$.  We take this divisor to be $(O) + (P)$, where $O$ is the identity element of the group law of $E$ and $P$ is another $\F_q$-rational point of $E$.  

A point $P \in E(\F_q)$ gives a map from $E$ to $\Proj^1$ by taking sections of the divisor $(O)+(P)$.  A root of this quartic corresponds to a point $Q \in E(\overline{\F}_q)$ with $2Q \sim O + P$, or $2Q = P$ in the group law on the curve.

We vary over all choices of $P$ and consider how many $Q$ occur as points with $2Q = P$.  If $\#E(\F_q)$ is odd, then the map $P \rightarrow 2P$ is an isomorphism, so every $P$ gives exactly one such $Q$.  If $\#E(\F_q)$ is even then there are two possibilities for the group structure of $E(\F_q)[2]$.  If $E(\F_q)[2] \cong \Z/2\Z$ then $1/2$ of points of $E(\F_q)$ have $0$ preimages under the map $P \rightarrow 2P$, and $1/2$ have exactly $2$.  If ${E(\F_q)[2] \cong \Z/2\Z \times \Z/2\Z}$ then $1/4$ of points of $E(\F_q)$ have $4$ preimages under the map $P \rightarrow 2P$, and $3/4$ have none. 
\end{proof}

We can now state the main result about $\QR^S_{C_{1,4}}(X,Y,Z)$.  
\begin{thm}\label{QRWt}
Suppose $q = p^v$ where $p$ is an odd prime.  Then $\QR^{S}(X,Y,Z)$ is equal to $(q-1)^2 q (q+1)$ times
\begin{eqnarray*}
& & \sum_{\stackrel{t^2\leq 4q}{t\equiv 1\pmod*{2}}} N_A(t)  X Y^{\frac{q-t}{2}} Z^{\frac{q+t}{2}} \\
& & + \sum_{\stackrel{t^2\leq 4q}{t\equiv 0\pmod*{2}}}  \bigg( (N_A(t) - N_{A,2\times 2}(t))  \left(\frac{1}{2}X^2 Y^{\frac{q-1-t}{2}} Z^{\frac{q-1+t}{2}} + \frac{1}{2}Y^{\frac{q+1-t}{2}} Z^{\frac{q+1+t}{2}}\right) \\
& & +  N_{A,2\times2}(t) 
\left(\frac{1}{4} X^4 Y^{\frac{q-3-t}{2}} Z^{\frac{q-3+t}{2}} + \frac{3}{4}Y^{\frac{q+1-t}{2}} Z^{\frac{q+1+t}{2}}\right)\bigg).
\end{eqnarray*}
\end{thm}

Combining Proposition \ref{QRsing}, Theorem \ref{S46withweights}, Lemma \ref{Schoof1}, and Theorem \ref{QRWt} completely determines $\QR_{C_{1,4}}(X,Y,Z)$.

We also give the quadratic residue weight enumerator of $C'_{1,h}$, the classical Reed-Solomon code of order $h$, as it can be computed easily from the analogous weight enumerator of  $C_{1,h}$.  Recall that we get the classical Reed-Solomon code by puncturing the projective one at a single point, that is, we choose one of the $q+1$ coordinates of our code and consider the image of the map that takes a codeword to the element of $\F_q^q$ that comes from deleting this coordinate.

\begin{prop}\label{RSWt}
Suppose that $h \le q$ and that the quadratic residue weight enumerator of $C_{1,h}$ is given by
\[
\QR_{C_{1,h}}(X,Y,Z) = \sum_{ j,k \geq 0 \atop j+k \le q+1} A_{q+1-j-k, j, k} X^{q+1-j-k} Y^j Z^k.
\]
Then the $X^{q-j-k} Y^j Z^k$ coefficient of the quadratic residue weight enumerator of the classical Reed-Solomon code of order $h$ is 
\[
\frac{(q+1 - j-k) A_{q+1-j-k,j,k} + (j+1) A_{q-j-k, j+1, k} + (k+1) A_{q-j-k,j,k+1}}{q+1},
\]
where $A_{i,j,k} = 0$ if any of $i,j,k < 0$.
\end{prop}
\begin{proof}
We consider all of the $A_{q+1-j-k, j, k}$ codewords of $C_{1,h}$ that have exactly $q+1-j-k$ coordinates equal to zero, $j$ equal to non-zero quadratic residues, and $k$ equal to non-zero quadratic non-residues.  The automorphism group of $\Proj^1(\F_q),\ \PGL_2(\F_q)$, is transitive on $\F_q$-points, so the proportion of these codewords that have a zero in a chosen coordinate is $\frac{q+1-j-k}{q+1}$.  Similar computations give the other two terms in the sum. 
\end{proof}

\subsection{The MacWilliams theorem for the quadratic residue weight enumerator}

Now that we have an expression for $\QR_{C_{1,4}}(X,Y,Z)$ we use a variation of the MacWilliams theorem to derive formulas for the coefficients of $\QR_{C_{1,4}^\perp}(X,Y,Z)=\QR_{C_{1,q-5}}(X,Y,Z)$. 
\begin{thm}\label{QRMac}
Let $C\subseteq \F_q^N$ be a linear code where $q $ is an odd prime power. Then 
\[
\QR_C(X,Y,Z) = \frac{1}{|C^\perp|} \QR_{C^\perp}(X',Y',Z'),
\]
where 
\begin{eqnarray*}
X' & = &  X+ \frac{q-1}{2}(Y+Z),\\
Y' & = & X + \frac{-(Y+Z) + \epsilon_q \sqrt{q} (Y-Z)}{2},\\
Z' & = &  X + \frac{-(Y+Z) + \epsilon_q \sqrt{q} (Z-Y)}{2},
\end{eqnarray*}
and
\[
\epsilon_q = 
\begin{cases} 1 &\mbox{if } q \equiv 1 \pmod{4} \\
i & \mbox{if } q \equiv 3 \pmod{4}. \end{cases}
\]
\end{thm}

We first recall the definition of the complete weight enumerator and the MacWilliams theorem for it and then prove Theorem \ref{QRMac} by specializing certain variables.  We follow the construction as described in Chapter 5 of \cite{MacwilliamsSloane}.  Let $w_0, w_1,\ldots, w_{q-1}$ be an enumeration of the elements of $\F_q$ with $w_0 = 0$.  The composition of $c = (c_1,\ldots, c_n) \in \F_q^n$, is defined by $\comp(c) = (s_0, s_1,\ldots, s_{q-1})$ where $s_i = s_i(c)$ is the number of coordinates $c_j$ equal to $w_i$.  We consider the additive group algebra $\C[\F_q]$.  In $\C[\F_q]$ we denote the elements of $\F_q$ by $z_i$ where $z_i$ corresponds to $w_i$.  

For $c = (c_1,\ldots, c_N) \in \F_q^N$, let $F(c) =  z_0^{s_0} z_1^{s_1}\cdots z_{q-1}^{s_{q-1}}$ where $\comp(c) = (s_0,\ldots, s_{q-1})$. The complete weight enumerator of $C$ is
\[
\CW_C(z_0,z_1,\ldots, z_{q-1}) \eqdef \sum_{c \in C} F(c) = \sum_{s= (s_0,\ldots, s_{q-1})} A(s) z_0^{s_0} \cdots z_{q-1}^{s_{q-1}},
\]
where $A(s)$ is the number of $c\in C$ with $\comp(c) = (s_0,\ldots, s_{q-1})$.

We recall some basic facts about characters on $\F_q$.  Suppose $q= p^v$ for $p$ prime.  There is an element $\beta \in \F_q$ such that $\{1,\beta,\beta^2,\ldots, \beta^{v-1}\}$ is a basis for $\F_q$ as a vector space over $\F_p$.  We uniquely identify the element 
\[
\gamma = \gamma_0 + \gamma_1 \beta +\cdots + \gamma_{v-1} \beta^{v-1}
\]
by $(\gamma_0,\ldots, \gamma_{v-1})$. Let $\zeta = e^{2\pi i/p}$ and $\chi: \F_q \rightarrow \C$ be defined by 
\[
\chi(\gamma) = \zeta^{\gamma_0 + \gamma_1 + \cdots + \gamma_{v-1}}
\]
for $\gamma = (\gamma_0,\ldots, \gamma_{v-1}) \in \F_q$.  This is an additive character of $\F_q$.  The following version of the MacWilliams theorem for the complete weight enumerator is Theorem 10 in Chapter 5 of \cite{MacwilliamsSloane}.
\begin{thm}{\label{CompMac}}
Let $C\subset \F_q^N$ be a linear code and $\chi$ the additive character on $\F_q$ defined above. Then $\CW_{C^\perp}(z_0,\ldots, z_{q-1})$ is given by
\[
\frac{1}{|C|} \CW_C\left(\sum_{j=0}^{q-1} \chi(w_0 w_j) z_j, \sum_{j=0}^{q-1} \chi(w_1 w_j) z_j, \ldots, \sum_{j=0}^{q-1} \chi(w_{q-1} w_j) z_j\right).
\]
\end{thm}

Recall that for the quadratic character $\eta$ on $\F_q^*$ we have
\[
\sum_{x\in \F_q^*} (1+\eta(x))\chi(x) = 2 \sum_{x\in \left(\F_q^{*}\right)^2} \chi(x).
\] 

The following is Theorem 5.15 in \cite{LidlNiederreiter}.
\begin{lemma}\label{GaussSum}
Suppose that $q =p^v$ is odd where $p$ is an odd prime and $v \ge 1$.  Let $\chi$ and $\eta$ be defined as above. Then 
\[
\sum_{x\in \F_q^*} \eta(x) \chi(x) = \epsilon_q \sqrt{q},\ \text{ where } \epsilon_q = 
\begin{cases} 
(-1)^{v-1} \ & \text{if } p \equiv 1\pmod{4} \\ 
(-1)^{v-1} i^v \ & \text{if } p \equiv 3 \pmod{4}. 
\end{cases}
\]
\end{lemma}

We now prove Theorem \ref{QRMac}.
\begin{proof}
Suppose we are in the setting of Theorem \ref{CompMac}.  Let $z_0 = X,\ z_i = Y$ if $w_i$ is a non-zero quadratic residue, and $z_i = Z$ otherwise.  We first note that
\[
 \sum_{i=0}^{q-1} \chi(w_0 w_i) z_i = X + \frac{q-1}{2}\left(Y+Z\right),
\]
and that for $w_j \in  \left(\F_q^*\right)^2$ we have
\[
\sum_{i=0}^{q-1} \chi(w_j w_i) z_i = X + Y \sum_{x \in \left(\F_q^*\right)^2} \chi(x) + Z \sum_{x \in \F_q^* \setminus \left(\F_q^*\right)^2} \chi(x).
\]
Since $\sum_{x\in \F_q^*} \chi(x) = -1$, applying Lemma \ref{GaussSum} gives 
\[
\sum_{x \in \left(\F_q^*\right)^2} \chi(x)  =\frac{-1 + \epsilon_q \sqrt{q}}{2}
\]
and
\[
\sum_{x \in \F_q^* \setminus \left(\F_q^*\right)^2} \chi(x) = \frac{-1 - \epsilon_q \sqrt{q}}{2},
\] 
where $\epsilon_q$ is defined as above.  When $w_j \in \F_q^* \setminus \left(\F_q^*\right)^2$ the coefficients of $Y$ and $Z$ are switched.  As $\QR_{C}(X,Y,Z)$ is symmetric in $Y,Z$, we may drop the negative signs in the definition of $\eps_q$.  Combining these observations completes the proof.
\end{proof}

\section{Point count distributions for elliptic curves over $\F_q$}\label{PointCountDistributions}

Let $t_E= q+1-\#E(\F_q)$ be the trace of Frobenius associated to $E$, and recall the definition of $N_A(t)$ from Section \ref{genus1curves}.  
Let \est{S_R^*(q) \eqdef \sum_{E/\F_q} \frac{t_E^{2R}}{|\Aut_{\F_q} (E)|} =  \sum_{t^2\leq4q} t^{2R} N_A(t) } be the (weighted) $2R$\textsuperscript{th} moment of $\#E(\F_q)$ over $\F_q$-isomorphism classes of elliptic curves $E/\F_q$.  

Let $\tr_{\SL_2(\Z),k}T_q$ be the trace of the $T_q$ Hecke operator acting on the space of weight $k$ cusp forms for the full modular group and $C_R = (2R)!/(R!(R+1)!) $ be the $R$\textsuperscript{th} Catalan number.  We also define the following combinatorial coefficients that show up repeatedly in our formulas 
\est{ a_{R,j} \eqdef \frac{2R-2j+1}{2R+1}\binom{2R+1}{j} =\binom{2R}{j} - \binom{2R}{j-1} .} 
In particular, we have $a_{R,R}= C_R$ and $a_{R,0}=1.$

If $q=p^v$ is a prime power we define 
\est{\rho(q,k) \eqdef  -\tr_{\SL_2(\Z),k}T_q + \frac{k-1}{12} q^{k/2-1} \mathds{1}_{v\equiv 0 \pmod* 2} -\frac{1}{2} \sum_{0 \leq i \leq v } \min(p^i,p^{v-i})^{k-1} \\ + \sigma_1(q)\mathds{1}_{k=2},}  
where 
$\mathds{1}_A$ is the indicator function of $A$ being true and $\sigma_1$ is the sum-of-divisors function, that is, $\sigma_1(n) = \sum_{d\mid n } d$.  Furthermore we set \est{\rho(1,k) \eqdef \frac{i^{k-2}}{4}+\frac{1}{3} \frac{\omega^{k-1}-\overline{\omega}^{k-1}}{\omega-\overline{\omega}} \,\,\,\text{ and }\,\,\, \rho(p^{-1},k) \eqdef 0,} where $\omega$ is a primitive $3$\textsuperscript{rd} root of unity.

\begin{thm}\label{birch} We have for prime $p\geq 3$ that 
\begin{eqnarray*}
S_0^*(p) &= & p \\ S_1^*(p) & = & p^2-1 \\ S_2^*(p) & = & 2p^3-3p-1 \\ S_3^*(p) & = & 5p^4-9p^2-5p-1 \\ S_4^*(p) & = & 14p^5-28p^3-20p^2 -7p-1 \\ S_5^*(p) & = & 42p^6 -90p^4-75p^3 -35p^2-9p-1 -\tau(p),
\end{eqnarray*} 
where $\tau(p)$ is Ramanujan's $\tau$-function.  
In general if $q=p^v$ with $p \neq 2$ we have \est{S^*_{R}(q) =\sum_{j=0}^R a_{R,j}  q^j \left( \rho(q,2R-2j+2)-p^{2R-2j+1}\rho(q/p^2,2R-2j+2)\right) \\ + \frac{p-1}{12}(4q)^R\mathds{1}_{v\equiv 0 \pmod* 2}.}  
\end{thm}
\begin{proof} Theorem \ref{birch} is due to Birch \cite{Birch} in the prime field case.  The generalization to all finite fields is well-known, being implicit in the work of Ihara \cite{Ihara}. See also Section 3 of the paper of Brock and Granville \cite{BrockGranville}. 
\end{proof}

For our application to computing the coefficients of the quadratic residue weight enumerator of the codes $C_{1,q-5}$ we prove the following variation of Birch's theorem to elliptic curves with rational 2-torsion. 

Let 
\est{ S_{2,R}^*(q) \eqdef \sum_{\substack{E/\F_q \\ E(\F_q)[2]\neq \{O\}}} \frac{t_E^{2R}}{|\Aut_{\F_q} (E)|} =  \sum_{\substack{t^2\leq4q \\ t \equiv 0 \pmod* 2}} t^{2R} N_A(t) }
be the (weighted) $2R$\textsuperscript{th} moment of the number of rational points of isomorphism classes of elliptic curves over $\F_q$ with at least one non-zero rational 2-torsion point.  

Let $\tr_{\Gamma_0(4),k}T_q$ be the trace of the Hecke operator $T_q$ acting on the space of classical cusp forms $S_k(\Gamma_0(4))$ of weight $k$ for the congruence subgroup $\Gamma_0(4)$, similarly for the congruence subgroup $\Gamma_0(2)$.

If $q=p^v$ is a prime power we define 
\est{\tau(q,k) \eqdef   \frac{k-1}{12} q^{\frac{k}{2}-1}\mathds{1}_{v \equiv 0 \pmod* 2} + \frac{1}{3} \tr_{\Gamma_0(4),k} T_q  -  \tr_{\Gamma_0(2),k} T_q \\  - \frac{1}{2}\sum_{0 \leq i \leq v}\min(p^i,p^{v-i})^{k-1}+ \frac{2}{3}\sigma_1(q)\mathds{1}_{k=2}.}  
Furthermore we set \est{\tau(1,k) \eqdef \frac{i^{k-2}}{4}\,\,\,\text{ and }\,\,\, \tau(p^{-1},k) \eqdef0.}  

\begin{thm}\label{Birchwith2tors}
If $q=p$ is an odd prime we have 
\begin{eqnarray*}
S^*_{2,0}(p) & = & \frac{1}{3} (2p-1) \\ S^*_{2,1}(p) & = & \frac{1}{3}p(2p-1)-1 \\ S^*_{2,2}(p) & = & \frac{4}{3} p^3 - \frac{2}{3} p^2 -3 p -1+ \frac{1}{3}a(p),
\end{eqnarray*}
where $a(p)$ is the $p$\textsuperscript{th} Fourier coefficient of $\eta^{12}(2z)$, the unique normalized Hecke eigenform of weight $6$ for $\Gamma_0(4)$.     
In general if $q=p^v$ with $p$ an odd prime we have 
\est{S^*_{2,R}(q) =\sum_{j=0}^R a_{R,j}  q^j \left( \tau(q,2R-2j+2)-p^{2R-2j+1}\tau(q/p^2,2R-2j+2)\right) \\ + \frac{p-1}{12}(4q)^R\mathds{1}_{v\equiv 0 \pmod* 2}.}
\end{thm}
\noindent We prove this theorem in Section \ref{proof}.

Finally, let
\est{S_{2\times 2,R}^*(q) \eqdef \sum_{\substack{E/\F_q \\ E(\F_q)[2] \cong \Z/2\Z \times \Z/2\Z}} \frac{t_{E}^{2R}}{|\Aut_{\F_q} (E)|} = \sum_{\substack{t \equiv q+1 \pmod* 4  \\ t^2\leq 4q}} t^{2R}N_{A,2\times 2}(t)} 
be the weighted $2R$\textsuperscript{th} moment of elliptic curves over $\F_q$ with full rational $2$-torsion.  
Let 
\est{ \phi(q,k) \eqdef -\frac{1}{6} \tr_{\Gamma_0(4),k} T_q+ \frac{k-1}{12} q^{k/2-1} \mathds{1}_{v\equiv 0\pmod* 2} -\frac{1}{4} \sum_{0\leq i \leq v} \min(p^i,p^{v-i})^{k-1} \\ +\frac{1}{6} \sigma_1(q)\mathds{1}_{k=2}. } 
Furthermore we set \est{\phi(1,k) = \phi(p^{-1},k) \eqdef 0. }  

\begin{thm}\label{Birchwithfull2tors}
If $q=p$ is an odd prime we have 
\begin{eqnarray*} 
S^*_{2\times 2,0}(p) & = & \frac{1}{6}p -\frac{1}{3} \\ S^*_{2\times 2,1}(p) & = & \frac{1}{6}p^2 -\frac{1}{3}p -\frac{1}{2}  \\ S^*_{2 \times 2,2}(p) & = & \frac{1}{3} p^3-\frac{2}{3} p^2 -\frac{3}{2}p -\frac{1}{2} -\frac{1}{6} a(p),
\end{eqnarray*}
where $a(p)$ is the $p$\textsuperscript{th} Fourier coefficient of $\eta^{12}(2z)$, the unique normalized Hecke eigenform of weight $6$ for $\Gamma_0(4)$.    
In general if $q=p^v$ with $p$ an odd prime we have 
\est{S^*_{2 \times 2,R}(q) =\sum_{j=0}^R a_{R,j}  q^j \left( \phi(q,2R-2j+2)-p^{2R-2j+1}\phi(q/p^2,2R-2j+2)\right) \\ + \frac{p-1}{12}(4q)^R\mathds{1}_{v\equiv 0 \pmod* 2}.} 
\end{thm}

\begin{proof} Theorem \ref{Birchwithfull2tors} is essentially due to Ahlgren \cite{AhlgrenFivefold} in the prime field case.  The generalization to all finite fields follows the same lines as Theorem \ref{Birchwith2tors} so we omit it.\end{proof} 

Remarks:
\begin{enumerate}
\item Our results re-prove the ``vertical'' Sato-Tate law for elliptic curves with specified rational $2$-torsion, which is of course already known in much greater generality, see e.g.~\cite{KatzSarnak}. On the other hand we believe the full formula for the moments in terms of traces of Hecke operators to be new and interesting in its own right.
\item The case $R=0$ of Theorem \ref{Birchwith2tors} is a special case of work of Eichler from the 1950s \cite{EichlerRelation}, and is elementary.  A nice exposition is given by Moreno \cite{Moreno}, see Theorem 5.10.  
\item Theorem \ref{birch} can be understood in terms of counting rational points on fibered products of the universal elliptic curve.  Work of Deligne and others relate this problem to cohomology groups of Kuga-Sato varieties.  For an introduction to these ideas see the book \cite{EdCo}, particularly Proposition 1.5.12 and Section 2.4.  It is likely that our results can be interpreted in this setting by taking fibered products of modular curves of level $2$, however we have opted instead for arguments similar in spirit to Birch's original proof.
\item Note that the coefficient of $\tr_{\Gamma_0(4),k} T_q$ in $\tau(q,k)$ and $\phi(q,k)$ differs by exactly a factor of $-1/2$.  We will use this fact crucially in the following proof.    
\end{enumerate}

\begin{proof}[Proof of Theorem~{\rm\ref{MainTheorem}}]
Recall the definitions of $M_v$ and $M_{v,R}(N)$ for $N=1,2,4$ from just above the statement of Theorem \ref{MainTheorem} in Section \ref{intro}.  For notational convenience, let 
\est{a= & \frac{ -1 + \epsilon_q \sqrt{q}}{2} \\ \overline{a} = &  \frac{ -1 - \epsilon_q \sqrt{q}}{2} } 
and note that $a+\overline{a}=-1$.   
The MacWilliams substitution cf. Theorem \ref{QRMac} is \est{ X & \mapsto X+\frac{q-1}{2} \left(Y+Z\right) \\ Y & \mapsto X+ aY+\overline{a} Z \\ Z & \mapsto X+\overline{a}Y+aZ.}  We use $\mapsto$ to denote this substitution below.  As a preliminary step we prove the following weaker version of Theorem \ref{MainTheorem}.  

\begin{lemma}\label{MTweak}
For each fixed $v,j,k$ and $\epsilon \in \{\pm 1\}$, there exists
\est{h \in M_v +  M_{v,\lfloor \frac{j+k}{2}\rfloor}(1)+ M_{v,\lfloor \frac{j+k}{2}\rfloor}(2)+ M_{v,\lfloor \frac{j+k}{2}\rfloor}(4)}
such that $A_{q+1-j-k,j,k} = h(q)$ for all $v$\textsuperscript{th} powers of odd primes $q = p^v$ satisfying $q\equiv \epsilon \pmod 4$.
\end{lemma} 
\begin{proof}  By Proposition \ref{QRsing}, the substitution $\mapsto$ applied to $\QR^{\text{sing}}_{C_{1,4}}(X,Y,Z)$ produces only polynomials as coefficients.  Thus we need only consider the smooth part, $\QR^S_{C_{1,4}}(X,Y,Z)$.  We only consider the first term in Theorem \ref{QRWt},\es{\label{lemmasum}{(q-1)^2 q (q+1)} X\sum_{\substack{t^2\leq 4q \\ t\equiv 1\pmod*{2} }} N_A(t)  Y^{\frac{q-t}{2}} Z^{\frac{q+t}{2}},} 
the other terms being treated similarly.

We apply the MacWilliams substitution and the trinomial expansion to \eqref{lemmasum}.  We write the trinomial coefficients as 
\est{\binom{n}{a,b} \eqdef \frac{\Gamma(n+1)}{\Gamma(n+1-a-b)\Gamma(a+1)\Gamma(b+1)}.} 
This definition makes sense even when $a, b$ or $n-(a+b)$ is negative.  Then 
\es{\label{A1}Y^{\frac{q-t}{2}}Z^{\frac{q+t}{2}} \mapsto \\ \sum_{j_Y,j_Z,k_Y,k_Z} \binom{\frac{q-t}{2}}{j_Y,j_Z}\binom{\frac{q+t}{2}}{k_Y,k_Z} X^{q-j_Y-j_Z-k_Y-k_Z}(aY)^{j_Y}(\overline{a}Y)^{k_Y}(aZ)^{k_Z}(\overline{a}Z)^{j_Z}} 
and we multiply this expression by \est{X \mapsto X+\frac{q-1}{2}\left(Y+Z\right).}  
Expanding the result we have the coefficient of $X^{q+1-j-k}Y^jZ^k$ is 
\es{\label{sumexpression}\sum_{\substack{j_Y+k_Y = j \\ j_Z + k_Z = k}} \binom{\frac{q-t}{2}}{j_Y,j_Z}\binom{\frac{q+t}{2}}{k_Y,k_Z} a^{j_Y + k_Z} \overline{a}^{k_Y + j_Z} 
\\ +  \frac{q-1}{2}\sum_{\substack{j_Y+k_Y +1 = j \\ j_Z + k_Z = k}} \binom{\frac{q-t}{2}}{j_Y,j_Z}\binom{\frac{q+t}{2}}{k_Y,k_Z} a^{j_Y + k_Z} \overline{a}^{k_Y + j_Z}
\\+ \frac{q-1}{2}\sum_{\substack{j_Y+k_Y  = j \\ j_Z + k_Z +1= k}}\binom{\frac{q-t}{2}}{j_Y,j_Z}\binom{\frac{q+t}{2}}{k_Y,k_Z} a^{j_Y + k_Z} \overline{a}^{k_Y + j_Z}.}  
 For each fixed $j_Y,j_Z,k_Y,k_Z$ the product of two trinomial coefficients here is a polynomial in $q$ and $t$ of degree at most $j_Y+j_Z+k_Y+k_Z$ in $t$, all but finitely many of which are 0.  We need consider only the terms of the sum \eqref{sumexpression} which are of even degree in $t$ since the odd degree terms are all killed by the sum over $t$ later. Note then that for each fixed $j,k$ the even degree part of the sum \eqref{sumexpression} is a polynomial in $q$ and $t$. We call this polynomial $p_{j,k}(t,q).$  The coefficient of $X^{q+1-j-k}Y^jZ^k$ in the MacWilliams substitution applied to \eqref{lemmasum} is therefore 
\est{(q-1)^2 q (q+1) \sum_{\substack{t^2\leq 4q \\ t\equiv 1\pmod*{2} }} N_A(t) p_{j,k}(t,q),} 
and we may form expressions for these coefficients in terms of the sums $S^*_R(q)-S^*_{2,R}(q)$ for $0\leq R \leq \lfloor \frac{j+k}{2}\rfloor.$  
 
Applying Theorems \ref{birch} and \ref{Birchwith2tors} we get a formula for the coefficient of $X^{q+1-j-k}Y^jZ^k$ in the expression which results from applying $\mapsto$ to \eqref{lemmasum} involving polynomials in $p$, $\rho(q,k),\rho(q/p^2,k),\tau(q,k)$, and $\tau(q/p^2,k)$.  The other terms in Theorem \ref{QRWt} are treated similarly, using Theorems \ref{Birchwith2tors} and \ref{Birchwithfull2tors}. \end{proof}  

Now we proceed to prove the stronger statement of Theorem \ref{MainTheorem}.  Take the expression in Theorem \ref{QRWt} and consider it now not as a polynomial, but as a real-analytic function on the open octant $\R_{>0}^3$.  We can now rearrange the expression found in Theorem \ref{QRWt} to find $\QR_{C_{1,4}}^S(X,Y,Z)$ is equal to \hbox{$(q-1)^2q(q+1)$} times \es{\label{QRWt2} X \sum_{t^2\leq 4q} N_A(t) Y^{\frac{q-t}{2}}Z^{\frac{q+t}{2}}   + \frac{1}{2} \left(X-Y^{1/2}Z^{1/2}\right)^2 \sum_{\substack{t\equiv 0 \pmod* 2 \\ t^2\leq 4q}} N_{A}(t) Y^{\frac{q-t-1}{2}}Z^{\frac{q+t-1}{2}} \\ + \frac{1}{4} \left(X^2-YZ\right)^2 \sum_{\substack{t\equiv 0 \pmod* 2 \\ t^2\leq 4q}} N_{A,2 \times 2}(t) Y^{\frac{q-t-3}{2}}Z^{\frac{q+t-3}{2}}.}

Let us call the three terms in \eqref{QRWt2} by $f_1(X,Y,Z), f_2(X,Y,Z)$ and $f_{2\times 2}(X,Y,Z)$, respectively.  The term $f_{2\times 2}$ is a polynomial in $X,Y,Z,$ but neither $f_1$ nor $f_2$ are polynomials.  Next we apply the MacWilliams substitution to \eqref{QRWt2}, giving three new real-analytic functions \est{g_1(X,Y,Z) \eqdef & f_1\left(X+\frac{q-1}{2}\left( Y+Z\right),X+ aY+\overline{a} Z,  X+\overline{a}Y+aZ\right) \\ g_2(X,Y,Z) \eqdef & f_2\left(X+\frac{q-1}{2}\left( Y+Z\right),X+ aY+\overline{a} Z,  X+\overline{a}Y+aZ\right) \\ g_{2\times 2}(X,Y,Z) \eqdef & f_{2\times 2}\left(X+\frac{q-1}{2}\left( Y+Z\right),X+ aY+\overline{a} Z,  X+\overline{a}Y+aZ\right) .} 

\begin{lemma}
Each of $g_1,g_2, $ and $g_{2\times 2}$ admits a convergent Laurent series in a neighborhood of $(X,Y,Z) = \left(\infty, 0,0\right)$.  
\end{lemma}  

\begin{proof} We take $X^{-1},Y,Z$ for our variables around $(\infty,0,0)$.  The lemma is clear for $g_{2 \times 2}$.  
To prove the lemma for $g_2$ it suffices to show that the MacWilliams substitution applied to $\left(X-Y^{1/2}Z^{1/2}\right)^2$ is a Laurent series in $X^{-1},Y,Z$. 
Indeed, using the power series expansion for $(1+u)^{1/2}$ about $u=0$, which is absolutely and uniformly convergent on compacts in $|u|<1$, we have \est{\left(YZ\right)^{1/2} \mapsto & X \left(1-\frac{Y+Z}{X} + \frac{(aY+\overline{a}Z)(\overline{a}Y+aZ)}{X^2}\right)^{1/2} \\ & = X\left(1- \frac{Y+Z}{2X} + O_{Y,Z}(X^{-2})\right).}  
Here $O_{Y,Z}(X^{-2})$ represents the higher order terms in this Laurent series expansion which have at least order $2$ in the variable $X^{-1}$ and unspecified orders in $Y$ and $Z$.  We thus have \est{\left(X-Y^{1/2}Z^{1/2}\right)^2 \mapsto \left(\frac{q}{2}\right)^2 \left(Y+Z\right)^2 + O_{Y,Z}(X^{-1}).} 

The term $g_1$ is treated similarly.  For any $t \in \Z$ with $t^2\leq 4q$ we have \est{  Y^{\frac{q-t}{2}}Z^{\frac{q+t}{2}} \mapsto X^q\left(1+ \frac{aY+\overline{a}Z}{X}\right)^{\frac{q-t}{2}} \left(1+ \frac{\overline{a}Y+aZ}{X}\right)^{\frac{q+t}{2}}.}  We again apply the power series expansion for $(1+u)^{1/2}$ around $u=0$ to get a convergent Laurent series expansion.  \end{proof}

By the lemma it suffices to study the coefficient of $(1/X)^{j+k-(q+1)}Y^jZ^k$ in the Laurent series expansion of each of $g_1,g_2$ and $g_{2\times2}$ separately, and show the sum of the three coefficients is in the module prescribed by the statement of Theorem \ref{MainTheorem}.     
Following the same technique as Lemma \ref{MTweak} we pick out the coefficient of $(1/X)^{j+k-(q+1)}Y^jZ^k$ from $g_1$ and apply Theorem \ref{birch} to evaluate $S^*_R(q)$.  Theorem \ref{birch} only yields polynomials in $p$ and traces of Hecke operators on spaces of cusp forms for $\SL_2(\Z)$, so there is nothing more to show concerning $g_1$.  

We extract the  coefficient of $(1/X)^{j+k-(q+1)}Y^jZ^k$ in the Laurent series expansion of $g_2$.  It is given by a sum over $t$, and we study the highest  power of $t$ in the summand.  By Theorems \ref{birch}, \ref{Birchwith2tors} and \ref{Birchwithfull2tors} this will give the highest weight trace of a Hecke operator possible in the final expression for $\QR_{C_{1,q-5}}(X,Y,Z)$.  
 
Following the same reasoning as in the proof of Lemma \ref{MTweak} we have 
\es{\label{A}Y^{\frac{q-t-1}{2}}Z^{\frac{q+t-1}{2}} \mapsto \\ \sum_{j_Y,j_Z,k_Y,k_Z} \binom{\frac{q-1-t}{2}}{j_Y,j_Z}\binom{\frac{q-1+t}{2}}{k_Y,k_Z} X^{q-1-j_Y-j_Z-k_Y-k_Z}(aY)^{j_Y}(\overline{a}Y)^{k_Y}(aZ)^{k_Z}(\overline{a}Z)^{j_Z}} which we multiply by \es{\label{B}\frac{1}{2}\left(\left(\frac{q}{2}\right)^2 \left(Y+Z\right)^2 + O_{Y,Z}(X^{-1}) \right).}  
The coefficient of $(1/X)^{j+k-(q+1)}Y^jZ^k$ is given by the terms of \eqref{A} satisfying 
\est{ j_Y+k_Y +2= j \\ j_Z+k_Z=k,} 
or \est{ j_Y+k_Y+1= j \\ j_Z+k_Z+1=k,} 
or \est{ j_Y+k_Y= j\\ j_Z+k_Z+2=k.} 
The highest power of $t$ appearing in each of the 3 cases is $j+k-2$.  Thus the coefficient of $(1/X)^{j+k-(q+1)}Y^jZ^k$ in the Laurent series for $g_2$ is given by polynomials in $q$ and $S^*_{2,R}(q)$ for $0\leq R\leq \lfloor \frac{j+k-2}{2}\rfloor$.  

We may apply the same reasoning to $g_{2\times 2}.$ We have \es{\label{C} \frac{1}{4}\left(X^2-YZ\right)^2 \mapsto 
\frac{1}{4}q^2(Y+Z)^2X^2 + O_{Y,Z}(X),} where the $O_{Y,Z}(X)$ notation has the same meaning as before.  Note that the leading terms in \eqref{C} and \eqref{B} differ only by a factor of $2X^2$.  Applying the substitution $\mapsto$ to $Y^{\frac{q-t-3}{2}}Z^{\frac{q+t-3}{2}}$ and expanding with the trinomial expansion one gets an expression identical to \eqref{A1} but with each instance of $q-1$ replaced by $q -3$.  
As in the case of $g_2$, the highest power of $t$ in the coefficient of $(1/X)^{j+k-(q+1)}Y^jZ^k$ in $g_{2\times 2}$ is $j+k-2$.  Comparing \eqref{B} and \eqref{C} we see that the coefficient of $t^{j+k-2}$ within the coefficient of $(1/X)^{j+k-(q+1)}Y^jZ^k$ of $g_{2\times 2}$ differs from that of $g_2$ by exactly a factor of $2$ .

The term $\tr_{\Gamma_0(4),k}T_q$ appears in $\tau(q,k)$ and $\phi(q,k)$ with coefficients differing by a factor of $-1/2$.  Thus if $j+k-2$ is even then $\tr_{\Gamma_0(4),j+k}T_q$ always cancels out of the $(1/X)^{j+k-(q+1)}Y^jZ^k$ coefficient of $g_2+g_{2\times 2}$.  We have therefore that the Laurent series coefficients of $g_1+g_2+g_{2\times 2}$ lie in the prescribed module, and thus that the weight enumerator coefficients do as well.
 \end{proof}

\section{The Eichler-Selberg trace formula and the proof of theorem \ref{Birchwith2tors}}\label{proof}

Our main tool is the Eichler-Selberg trace formula.  
Our reference is Knightly and Li \cite{KnightlyLi} ``Statement of the final result''.   

Recall the Kronecker symbol $\left(\frac{\Delta}{n}\right)$, which we only use when $\Delta$ is a discriminant.  For $n$ an odd prime the Kronecker symbol is defined to be the quadratic residue symbol and for $n=2$ we define \est{\left(\frac{\Delta}{2}\right) \eqdef \begin{cases} 0 & \text{ if } 2 \mid \Delta, \\ 1 & \text{ if } \Delta \equiv 1 \pmod 8, \\ -1 & \text{ if } \Delta \equiv 5 \pmod 8 .\end{cases} }
\begin{lemma}\label{cor728}
For $d<0$, $d\equiv 0,1\pmod 4$ and $f \in \N$ we have \est{h_w(f^2d) = h_w(d) f \prod_{p \mid f} \left(1 - \legen{d}{p}\frac{1}{p}\right).}
\end{lemma}
\begin{proof} This is a standard result.  See e.g. Corollary 7.28 of \cite{primesoftheform}.  \end{proof}

\begin{prop}[ESTF for odd prime powers $q$ and $\Gamma_0(2)$]\label{gamma2}
Let $q=p^v$ be an odd prime power.  Let $t\in \Z$ run over integers such that $t^2< 4q$. Let $\alpha$ and $\overline{\alpha}$ be the two roots of $X^2-tX+q=0$ in $\C$. 

We have  
\est{\tr_{\Gamma_0(2), k} T_q = &\frac{k-1}{4} q^{\frac{k}{2}-1}\mathds{1}_{v \equiv 0 \pmod* 2} \\ & -\frac{1}{2} \sum_{t\equiv 0 \pmod* 2} \frac{\alpha^{k-1}- \overline{\alpha}^{k-1}}{\alpha-\overline{\alpha}} \sum_{\substack{m^2 \mid t^2-4q \\ \frac{t^2-4q}{m^2} \equiv 0,1 \pmod* 4 \\ \ord_2 m = 0}} h_w\left(\frac{t^2-4q}{m^2}\right) \\ 
& - \frac{3}{2} \sum_{t\equiv q+1 \pmod* 4} \frac{\alpha^{k-1}- \overline{\alpha}^{k-1}}{\alpha-\overline{\alpha}} H_w\left( \frac{t^2-4q}{4}\right)  \\ 
& -\sum_{0 \leq i \leq v}\min(p^i,p^{v-i})^{k-1} +  \sigma_1(q)\mathds{1}_{k=2}.} 
\end{prop}

\begin{proof}
Proposition \ref{gamma2} is a simplification of the standard Eichler-Selberg trace formula, where we have performed a careful but tedious case check of the behavior of the Hecke polynomial $X^2-tX+q$ modulo $2$ and $4$, and the discriminant $t^2-4q$ modulo $16$. The details are similar to but less complicated than those of the proof of Proposition \ref{gamma4} so we omit them. 
\end{proof}

\begin{prop}[ESTF for odd prime powers $q$ and $\Gamma_0(4)$]\label{gamma4}
Let $q=p^v$ be an odd prime power.  Let $t\in \Z$ run over integers such that $t^2< 4q$. Let $\alpha$ and $\overline{\alpha}$ be the two roots of $X^2-tX+q=0$ in $\C$. 

We have \est{\tr_{\Gamma_0(4),k} T_q= & \frac{k-1}{2} q^{\frac{k}{2}-1}\mathds{1}_{v \equiv 0 \pmod* 2}  -3 \sum_{t\equiv q+1 \pmod* 4} \frac{\alpha^{k-1}- \overline{\alpha}^{k-1}}{\alpha-\overline{\alpha}} H_w\left( \frac{t^2-4q}{4}\right) \\ & - \frac{3}{2}\sum_{0 \leq i \leq v}\min(p^i,p^{v-i})^{k-1}+ \sigma_1(q)\mathds{1}_{k=2}.} 
\end{prop}

\begin{proof}
Apart from trivial simplifications the formula above differs from that appearing in Knightly and Li \cite{KnightlyLi} only by the weights appearing in the sum over $t$.  Specifically, to derive Proposition \ref{gamma4} from the standard formula in \cite{KnightlyLi} it suffices to show 
\est{ 3\cdot \mathds{1}_{t\equiv q+1\pmod* 4}  H_w\left( \frac{t^2-4q}{4}\right) = \frac{1}{2} \sum_{\substack{m^2 \mid t^2-4q \\ \frac{t^2-4q}{m^2} \equiv 0,1\pmod* 4}} h_w\left(\frac{t^2-4q}{m^2}\right)\mu(t,m,q),} where \est{\mu(t,m,q) \eqdef \frac{\psi(4)}{\psi(4/(4,m))} \sum_{c \in (\Z/4\Z)^\times} 1,
}  
$c$ runs through all elements of $(\Z/4\Z)^\times$ that lift to solutions of $c^2-tc+q\equiv 0 \pmod {4 (4,m)}, $ and 
\est{
\psi(N) \eqdef [\SL_2(\Z):\Gamma_0(N)] = N \prod_{\ell \mid N} \left(1+\frac{1}{\ell}\right).
}

We check the cases $(4,m)=1,2,4$ one-by-one and after a lengthy but trivial computation derive
\est{\mu(t,m,q) = \begin{cases} 2 \cdot \mathds{1}_{t\equiv 0 \pmod* 4} & \text{ if } q \equiv 3 \pmod 4 \text{ and } (4,m)=1 \\  
2 \cdot\mathds{1}_{t\equiv 2 \pmod* 4} & \text{ if } q \equiv 1 \pmod 4 \text{ and } (4,m)=1 \\  
2 \cdot\mathds{1}_{t\equiv 2 \pmod* 4} & \text{ if } q \equiv 1 \pmod 4 \text{ and } (4,m)=2 \\  
4\cdot\mathds{1}_{t\equiv 4 \pmod* 8} & \text{ if } q \equiv 3 \pmod 8 \text{ and } (4,m)=2 \\ 
4\cdot\mathds{1}_{t\equiv 0 \pmod* 8} & \text{ if } q \equiv 7 \pmod 8 \text{ and } (4,m)=2 \\ 
6\cdot\mathds{1}_{t\equiv \pm 2 \pmod* {16}} & \text{ if } q \equiv 1,13 \pmod {16} \text{ and } (4,m)=4 \\  
6\cdot\mathds{1}_{t\equiv \pm 6 \pmod* {16}} & \text{ if } q \equiv 5,9 \pmod {16} \text{ and } (4,m)=4 \\ 
\text{does not occur} & \text{ if } q\equiv 3 \pmod 4 \text{ and } (4,m)=4. \end{cases}} If $q \equiv 3 \pmod 4$ then $t^2-4q$ can only take the values $4,5,8,13 \pmod {16}$, hence the last entry above.

In all cases above we have $4 \mid t^2-4q$, so we use Lemma \ref{cor728} to re-write the $\ord_2 m =0$ cases in the above in terms of $m$ with $\ord_2 m=1$.  First assume $q\equiv 3 \pmod 4$, and that $\ord_2 m = 0$.  We know $\ord_2 (t^2-4q)=2$ and $t\equiv 0 \pmod 4$ so that 
\est{ \frac{t^2-4q}{4m^2} \equiv \begin{cases} 1 \pmod 8 & \text{ if } t+q \equiv 7 \pmod 8 \\ 5 \pmod 8 & \text{ if } t+q \equiv 3 \pmod 8. \end{cases}} By Lemma \ref{cor728} if $\ord_2 m=0$ and $q \equiv 3 \pmod 4$ then \est{ h_w\left(\frac{t^2-4q}{m^2}\right)=h_w\left(\frac{t^2-4q}{4m^2} \right) \begin{cases} 1 & \text{ if } t+q \equiv 7 \pmod 8 \\ 3 & \text{ if } t+q \equiv 3 \pmod 8. \end{cases} }

Now we turn to the $q\equiv 1\pmod 4$ and $\ord_2 m = 0$ case.  One easily checks that $t\equiv 2 \pmod 4$ and $q \equiv 1 \pmod 4$ imply $\ord_2(t^2-4q) \geq 3$ thus 
\est{h_w\left(\frac{t^2-4q}{m^2}\right)=2h_w\left( \frac{t^2-4q}{4m^2}\right).}
Collecting all the above terms, one arrives at the claimed formula.
\end{proof}

Recall the definition of $\tau(q,k)$ from just above Theorem \ref{Birchwith2tors}.

\begin{lemma}\label{MainProp}
Suppose $q$ is an odd prime power. Let $\alpha,\overline{\alpha} \in \C$ be solutions to $X^2-tX+q=0$ and $H_w(\Delta)$ as in Section \ref{genus1curves}. We have \est{ \frac{1}{2}\sum_{\substack{t\equiv 0 \pmod* 2 \\ t^2<4q}} \frac{\alpha^{k-1}- \overline{\alpha}^{k-1}}{\alpha-\overline{\alpha}} H_w(t^2-4q) = \tau(q,k).}
\end{lemma}
\begin{proof} 
Take $2/3$ times the formula of Proposition \ref{gamma4} minus $2$ times the result of Proposition \ref{gamma2}.  One finds that 
\begin{eqnarray*}
\tau(q,k) & = &   \frac{1}{2} \sum_{t\equiv 0 \pmod* 2} \frac{\alpha^{k-1}- \overline{\alpha}^{k-1}}{\alpha-\overline{\alpha}} \sum_{\ord_2 m = 0} h_w\left(\frac{t^2-4q}{m^2}\right)\\ 
& & +   \frac{1}{2}\sum_{t\equiv q+1 \pmod* 4} \frac{\alpha^{k-1}- \overline{\alpha}^{k-1}}{\alpha-\overline{\alpha}} \sum_{\ord_2 m\geq 1} h_w\left(\frac{t^2-4q}{m^2} \right).
\end{eqnarray*}
In each case $q \equiv 1,3 \pmod 4$ one checks that these $m$ give the complete list defining $H_w(t^2-4q)$, hence Lemma \ref{MainProp} follows.
\end{proof}

For general prime powers $q$ we need some additional definitions.  Let 
\est{\omega(q,k) \eqdef  \frac{1}{2} \sum_{\substack{t\equiv 0 \pmod* 2 \\ p \nmid t\\ t^2<4q}} \frac{\alpha^{k-1}- \overline{\alpha}^{k-1}}{\alpha-\overline{\alpha}} H_w(t^2-4q) = \sum_{\substack{t\equiv 0 \pmod* 2 \\ p \nmid t\\ t^2<4q}} \frac{\alpha^{k-1}- \overline{\alpha}^{k-1}}{\alpha-\overline{\alpha}} N_A(t)} 
be the contribution from ordinary elliptic curves.  We set 
\est{A_q \eqdef \frac{\alpha_0^{k-1}- \overline{\alpha_0}^{k-1}}{\alpha_0-\overline{\alpha_0}}} 
where $\alpha_0$ and $\overline{\alpha_0}$ be the two roots in $\C$ of $X^2+q=0$.  We also let \est{\omega'(q,k) \eqdef \frac{1}{2} \sum_{\substack{t\equiv 0 \pmod* 2 \\ p \nmid t\\ t^2<4q}} \frac{\alpha^{k-1}- \overline{\alpha}^{k-1}}{\alpha-\overline{\alpha}} H_w(t^2-4q) + A_q N_A(0) } be the contribution from elliptic curves whose ring of endomorphisms over $\F_q$ is an order in a quadratic field.  Lastly, recall we set \est{\tau(1,k) = \frac{i^{k-2}}{4} \,\,\, \text{ and } \,\,\, \tau(p^{-1},k)=0.}

\begin{lemma}\label{ord}  
Suppose $q = p^v$ is an odd prime power with $0\leq 2i < v$. We have 
\est{(p^i)^{k-1}\omega(q/p^{2i},k)  = \frac{1}{2} \sum_{\substack{t^2<4q \\ \ord_p t =i \\ t\equiv 0 \pmod* 2}} \frac{\alpha^{k-1}- \overline{\alpha}^{k-1}}{\alpha-\overline{\alpha}} H_w(t^2-4q) .}
\end{lemma}
\begin{proof} 
Observe that if $0 < i$
\est{ p^i = p^i\left(1-\frac{1}{p}\right) + p^{i-1} \left(1-\frac{1}{p}\right)  + \cdots + p\left(1-\frac{1}{p}\right) +1,}
and 
if $0\leq2i<v$ 
\est{ (p^i)^{k-1} \omega(q/p^{2i},k) = \frac{1}{2} \sum_{\substack{(p^it)^2<4q \\ p \nmid t \\ t \equiv 0 \pmod* 2}} \frac{\alpha^{k-1} - \overline{\alpha}^{k-1}}{\alpha-\overline{\alpha}} p^iH_w(t^2-4q/p^{2i})} 
with $\alpha, \overline{\alpha}$ being solutions in $\C$ to $X^2-(p^it)X+q$.  

Let $\Delta = t^2-4q/p^{2i} = ((p^i t)^2-4q)/p^{2i}.$ For such $\Delta$ we have that 
\est{p^i H_w(\Delta) = H_w(\Delta) + \sum_{j=1}^ip^j\left(1-\frac{1}{p}\right)H_w(\Delta).}  
Note that since $p \nmid t$ that $\Delta$ is a non-zero square modulo $p$.  Moreover, if $d^2 \mid \Delta$ then we also have that $p\nmid d$, and that $\Delta / d^2$ is a non-zero square modulo $p$.    The definition of $H_w$ and Lemma \ref{cor728} imply that 
\est{ p^iH_w(\Delta) & = \sum_{\substack{d'^2 \mid p^{2i}\Delta  \\ \frac{p^{2i}\Delta}{d'^2} \equiv 0,1(4) \\ \ord_p d' =i}} h_w\left( \frac{p^{2i}\Delta}{d'^2}\right) + \sum_{j=1}^i \sum_{\substack{d'^2 \mid p^{2i}\Delta \\ \frac{p^{2i}\Delta}{d'^2} \equiv 0,1(4) \\ \ord_p d' =i-j}} h_w\left( \frac{p^{2i}\Delta}{d'^2}\right) \\ & = H_w(p^{2i}\Delta)  = H_w\left((p^it)^2-4q\right).}  
Thus we have the lemma. 
\end{proof}

\begin{lemma}\label{ML}
Suppose $q = p^v$ is an odd prime power.  We have that \est{ \omega'(q,k) = \tau(q,k) - p^{k-1}\tau(q/p^2,k).}
\end{lemma}
\begin{proof}
Lemmas \ref{MainProp} and \ref{ord} allow us to write \est{ \tau(q,k) = \sum_{0\leq 2i<v} (p^i)^{k-1} \omega(q/p^{2i},k)+ \frac{1}{2} A_{q} H_{w}(-4q). }
Now we need to do some calculation with these class numbers.   The definition of $H_w$ and Lemma \ref{cor728} imply \est{ \frac{1}{2}H_w(-4q) = \begin{cases} \frac{\sqrt{q}
}{4} + \sigma_1(p^{v/2-1})N_A(0) & \text{ if } v\equiv 0 \pmod 2 \\ \sigma_1(p^{\frac{v-1}{2}}) N_A(0) & \text{ if } v \equiv 1 \pmod 2 ,\end{cases}} and we also have \est{ p^i A_{q} = p^i (i\sqrt{q})^{k-2} = (p^i)^{k-1} A_{q/p^{2i}} } from the definition.
It follows that 
\est{\label{tau}\tau(q,k) = \sum_{0\leq 2i<v}  (p^i)^{k-1} \omega'(q/p^{2i},k) + \frac{\sqrt{q}}{4}A_{q}\mathds{1}_{v\equiv 0 \pmod*{2}},} 
and changing variables we compute 
\est{p^{k-1}\tau(q/p^2,k) =\sum_{0< 2i<v}  (p^i)^{k-1} \omega'(q/p^{2i},k) +\frac{\sqrt{q}}{4}A_{q}\mathds{1}_{v\equiv 0 \pmod*{2}} .} 
By subtracting we prove the lemma.  
\end{proof}

Recall that we have set \est{ a_{R,j} = \binom{2R}{j}-\binom{2R}{j-1}.}
 
\begin{lemma}\label{swapwts}
We have \est{ t^{2R} = \sum_{j=0}^R a_{R,j} q^j \frac{\alpha^{2R-2j+1}-\overline{\alpha}^{2R-2j+1}}{\alpha-\overline{\alpha}}.}
\end{lemma}
\begin{proof} This is an easy proof by induction exercise.
\end{proof}
Putting together Lemmas \ref{ord} and \ref{swapwts} we get 
\est{\sum_{\substack{t\equiv 0 \pmod* 2 \\ t^2<4q }} t^{2R} N_A(t) = \sum_{j=0}^R a_{R,j}q^j \left( \tau(q,2R-2j+2)-p^{2R-2j+1}\tau(q/p^2,2R-2j+2)\right).} 
If $v$ is even we add the contribution from the terms corresponding to $t^2=4q$, i.e.,~those supersingular curves having ring of endomorphisms equal to a maximal order in a quaternion algebra over the base field.  This concludes the proof of Theorem \ref{Birchwith2tors}.

\section*{acknowledgements}
Part of this project grew out of the PhD thesis of the first author.  He thanks Noam Elkies for his extensive guidance and for many helpful conversations.  The authors also thank him for carefully reading a draft of this paper and the anonymous referee for helpful comments.

\def\cprime{$'$}
\bibliographystyle{amsplain}

\end{document}